\documentclass[12pt]{amsart}
\usepackage[T1]{fontenc}
\usepackage[latin1]{inputenc}
\usepackage{typearea}
\usepackage{geometry}
\usepackage{ulem}
\usepackage{amsmath}
\usepackage{amssymb}
\usepackage{latexsym}
\usepackage{enumerate}
\usepackage{amsthm}
\usepackage[all]{xy}
\usepackage{hhline}
\usepackage{epsf} 
\usepackage{cite}

\newtheorem{theorem}{{\sc Theorem}}[section]
\newtheorem{cor}[theorem]{{\sc Corollary}}

\newtheorem{lemma}[theorem]{{\sc Lemma}}
\newtheorem{prop}[theorem]{{\sc Proposition}}

\theoremstyle{remark}
\newtheorem{remark}[theorem]{{\sc Remark}}

\theoremstyle{definition}

\newcommand{\R}{\mathbb{R} }

\newcommand{\A}{\mathcal{A}}

\newcommand{\D}{\mathcal{D}}
\newcommand{\W}{\mathcal{W}}
\newcommand{\K}{\mathcal{K}}
\newcommand{\calH}{\mathcal{H}}
\newcommand{\calL}{\mathcal{L}}

\newcommand{\la}{\lambda}

\newcommand{\Om}{\Omega}

\providecommand{\abs}[1]{\lvert #1\rvert}

\providecommand{\fnorm}[1]{\lVert #1\rVert_\infty}
\DeclareMathOperator{\Var}{Var}

\DeclareMathOperator{\Exp}{Exp}

\DeclareMathOperator{\Laplace}{Laplace}

\DeclareMathOperator{\Poi}{Poisson}
\DeclareMathOperator{\Bin}{Binomial}

\renewcommand{\phi}{\varphi}
\renewcommand{\epsilon}{\varepsilon}

\renewcommand{\rho}{\varrho}

\begin{document}
\title[Rates of convergence for random sums]{On rates of convergence and Berry-Esseen bounds for random sums of centered random variables with finite third moments}
\author{Christian D\"obler}
\thanks{Technische Universit\"at M\"unchen, Fakult\"at f\"ur Mathematik, Lehrstuhl f\"ur Mathematische Physik, D-85740 M\"unchen, Germany. \\
christian.doebler@tum.de}
\begin{abstract}
We show, how the classical Berry-Esseen theorem for normal approximation may be used to derive rates of convergence for random sums of centerd, real-valued random variables with respect to a certain class of probability metrics, including the Kolmogorov and the Wasserstein distances. This technique is applied to several examples, including the approximation by a Laplace distribution of a geometric sum of centered random variables with finite third moment, where a concrete Berry-Esseen bound is derived. This bound reduces to a bound of the supposedly optimal order $\sqrt{p}$ in the i.i.d. case. 
\end{abstract}

\maketitle

\section{Introduction}\label{intro}
Random sums are random variables of the form 

\begin{equation}\label{ineq1}
S=\sum_{j=1}^N X_j\,,
\end{equation}

where $X_0,X_1,X_2,\dotsc$ and $N$ are random variables on a common probability space $(\Om,\A,P)$ and where 
$N$ assumes nonnegative integer values. Random variables of the form \eqref{ineq1} appear frequently in modern probabiliy theory, because many models for example from physics, finance, reliability and risk theory naturally lead to the consideration of such sums. Furthermore, sometimes a model, which looks completely different from \eqref{ineq1} at the outset, may be transformed into a random sum and then general theory of such sums may be invoked to study the original model \cite{GneKo}.\\ 
There already exists some literature on the asymptotic distributions of such random sums, see e.g. \cite{GneKo} for 
general distributions of the \textit{index} $N$ and \cite{Kalash} for the special case of \textit{geometric sums}, where $N=N_p$ has the geometric distribution on $\mathbb{Z}\cap[1,\infty)$ with parameter $p\in(0,1)$. Recall that this means that $P(N=k)=p(1-p)^{k-1}$ for $k=1,2,\dotsc$. In \cite{Kalash}, also quantitative theorems, giving bounds on the approximation error, are derived but mainly nonnegative summands $X_j$ are considered. Contrarily, the theorems from 
\cite{GneKo} are very general, not even the existence of moments is assumed at the outset, but are only qualitative in that they are not giving any error bounds. In \cite{PeRoe} \textit{Stein's method} for the exponential distribution together with a coupling construction, relying on the equilibrium transformation from renewal theory, 
are used to prove error bounds for the convergence of geometric (as well as more general) random sums of nonnegative (possibly dependent) random variables to the exponential distribution, generalizing a classical theorem by R\'{e}nyi. Recently, in \cite{PiRen} a suitable version of Stein's method for the Laplace distribution and a distributional transformation in the spirit of the theory from \cite{PeRoe} were developed to prove rates of convergence for geometric (as well as more general) sums of centered random variables with finite third moments to the Laplace distribution. At least in the special case of a geometric sum, the theorems from \cite{PiRen} seem to yield optimal rates of convergence in the bounded Lipschitz metric, but due to the distributional transformation, which is used for the coupling construction, the summands $X_j$ have to fullfil a certain symmetry condition in order that the coupling can be constructed.\\ 
The purpose of the present paper is to show how one can use the classical Berry-Esseen theorem for normal approximation along with asymptotic results on the index $N$ to obtain rates of convergence for suitably standardized versions of the sum \eqref{ineq1}. This is in a way analogous to the  \textit{transfer theorems} from \cite{GneKo}.\\
The paper is organized as follows. In Section \ref{setup} we present the general setup and some abstract results, which help bound the distance from (a scaled version of) the random variable $S$ to a suitable limit distribution. Then, in Section \ref{apps} we prove rates of convergence for several distributions of the index $N$, including the geometric, the Poisson and the binomial distributions.\\

\section{Abstract approximation results}\label{setup} 
Througout assume that the random variables $N,X_1,X_2,\dotsc$ are independent, that the $X_j$ are centered having finite third moments $\xi_j:=E\abs{X_j}^3$ and that 
$0<\mu:=E[N]<\infty$. Furthermore, we assume $0<\sigma_j^2:=E[X_j^2]=\Var(X_j)$, $j\geq1$, and define 

\begin{equation}\label{ateq1}
 W:=\mu^{-1/2}\sum_{j=1}^N X_j\,.
\end{equation}
 
Then, letting $s_N^2:=\Var(W|N)$ it is easy to see that $s_N^2=\mu^{-1}\sum_{j=1}^N \sigma_j^2$ and that $E[W|N]=0$. Furthermore, since 
$\Var(W)=E[\Var(W|N)]+\Var(E[W|N])$, it follows that 

\begin{align}\label{ateq2}
 \Var(W)&=E[s_N^2]=\mu^{-1}\sum_{n=0}^\infty\sum_{j=1}^n\sigma_j^2 P(N=n)\notag\\
&= \mu^{-1} \sum_{j=1}^\infty\sum_{n=j}^\infty P(N=n)= \mu^{-1} \sum_{j=1}^\infty P(N\geq j)\,.
\end{align}
 
Now, assume that $Z_1,Z_2,\dotsc$ are further random variables on the same probability space such that $Z_j\sim N(0,\sigma_j^2)$, $j\geq 1$, and that also 
$N,Z_1,Z_2,\dotsc$ are independent. Then, we define the random variable 

\begin{equation}\label{ateq3}
Z:=\mu^{-1/2}\sum_{j=1}^N Z_j\,.
\end{equation}

It is clear that $\calL(Z|N)=N(0,s_N^2)$ and hence, $\Var(Z|N)=s_N^2=\Var(W|N)$. For a given class $\calH$ of Borel-measurable functions $h:\R\rightarrow\R$ 
consider the distance 

\begin{equation}\label{ateq4}
 d_{\calH}(\mu,\nu):=\sup_{h\in\calH}\Bigl| \int_\R hd\mu-\int_\R hd\nu\Bigr|\,,
\end{equation}

on the space of probability measures with respect to which every $h\in\calH$ is integrable. For instance, if 
$\calH=\W$ is the class of $1$-Lipschitz functions, then \eqref{ateq4} reduces to the \textit{Wasserstein distance} of $\mu$ and $\nu$. On the other hand, 
if $\calH=\K=\{1_{(-\infty,z]}\,:\,z\in\R\}$, then $d_\K$ is called the \textit{Kolmogorov distance}. Both of these distances are in fact metrics, which yield stronger topologies on their domains than the topology induced by weak convergence (see, e.g. \cite{Dud02}). \\
For our approximation theorems we need the following result, which is an instance of the classical Berry-Esseen theorem 
(see, e.g.\cite{Fel2}).

\begin{prop}\label{berry}    
Under the above assumptions we have the following bounds: 
\begin{enumerate}[{\normalfont(a)}]
 \item $d_\W\bigl(\calL(W|N),N(0,s_N^2)\bigr)\leq\frac{C_\W}{s_N^2\mu^{3/2}}\sum_{j=1}^N \xi_j $
 \item $d_\K\bigl(\calL(W|N),N(0,s_N^2)\bigr)\leq\frac{C_\K}{s_N^3\mu^{3/2}}\sum_{j=1}^N \xi_j $
\end{enumerate}
Here, $C_\W$ and $C_\K$ are absolute and finite constants.
\end{prop}

\begin{remark}\label{berryrem}
 The search for the optimal value of the constant $C_\K$, the \textit{Berry-Esseen constant}, is still going on. As far as we know, the best proven bound is 
$C_\K\leq 0.56$ by Shevtsova \cite{Shev}.\\
Using Stein's method (see \cite{CheShaSing}), one may obtain $C_\W\leq6$, which seems to be far from optimal.   
\end{remark}

\begin{proof}[Proof of Proposition \ref{berry}]
Part (b) follows immediately from the classical Berry-Esseen theorem (see \cite{Fel2}) , since the Kolmogorov distance is scale-invariant. Since for random variables 
$X$ and $Y$ (with existing first moments) and $\sigma>0$ we have 

\[d_\W\bigl(\calL(\sigma X),\calL(\sigma Y)\bigr)=\sigma d_\W\bigl(\calL(X),\calL(Y)\bigr)\,,\]   
 
the claim of (a) follows from the Berry-Esseen result for the Wasserstein distance (see \cite{CheShaSing}, for example).\\
\end{proof}

Now, let $h\in\calH$ be a given test function. Then, 

\begin{align}\label{ateq5}
&\bigl|E[h(W)]-E[h(Z)]\bigr|=\bigl|E\bigl[E[h(W)|N]-E[h(Z)|N]\bigr]\bigr|\notag\\
&\leq E\bigl|E[h(W)|N]-E[h(Z)|N]\bigr|\notag\\
&\leq E\bigl|d_{\calH}\bigl(\calL(W|N),\calL(Z|N)\bigl)\bigr|\notag\\
&=E\Bigl[d_{\calH}\bigl(\calL(W|N),\calL(Z|N)\bigl)\Bigr]\,,
\end{align}

and hence, by taking the supremum over $h\in\calH$, we obtain from \eqref{ateq5} that

\begin{equation}\label{ateq6}
 d_{\calH}\bigl(\calL(W),\calL(Z)\bigr)\leq E\Bigl[d_{\calH}\bigl(\calL(W|N),\calL(Z|N)\bigl)\Bigr]\,.
\end{equation}

Recalling that $\calL(Z|N)=N(0,s_N^2)$ by Proposition \ref{berry} this leads to the following theorem.

\begin{prop}\label{attheo1}
Under our general assumptions, we have the following bounds:
\begin{enumerate}[{\normalfont(a)}] 
 \item $d_\W\bigl(\calL(W),\calL(Z)\bigr)\leq\frac{C_\W}{\mu^{3/2}}E\Bigl[\frac{1}{s_N^2}\sum_{j=1}^N \xi_j\Bigr]$
 \item  $d_\K\bigl(\calL(W),\calL(Z)\bigr)\leq\frac{C_\K}{\mu^{3/2}}E\Bigl[\frac{1}{s_N^3}\sum_{j=1}^N \xi_j\Bigr]$
\end{enumerate}
\end{prop}

\begin{cor}\label{atcor1}
Assume that, additionally to our general assumptions, the $X_j$ are identically distributed with $\sigma^2:=E[X_1^2]$ and $\xi:=E\abs{X_1}^3$. Then, 
 \begin{enumerate}[{\normalfont(a)}] 
 \item $d_\W\bigl(\calL(W),\calL(Z)\bigr)\leq\frac{C_\W\xi}{\mu^{1/2}\sigma^2}$ and
 \item  $d_\K\bigl(\calL(W),\calL(Z)\bigr)\leq\frac{C_\K\xi}{\sigma^3}E[\frac{1}{\sqrt{N}}]$.
\end{enumerate}
\end{cor}

Note that the bounds from Proposition \ref{attheo1} and from Corollary \ref{atcor1} tend to converge to zero as $\mu\to\infty$, i.e. as the number of summands is infinitely growing.\\
 The next step is to find bounds on the distance 
$d_{\calH}\bigl(\calL(Z),\calL(Y)\bigr)$ for some random variable $Y$. In order to motivate the concrete form of $Y$ let us for the moment assume that we are given a whole sequence $(N_k)_{k\geq1}$ of random indices, such that 
$\mu_k:=E[N_k]\to\infty$ as $k\to\infty$ and that $\mu_k^{-1}N_k\stackrel{\D}{\rightarrow} U$, where $U$ is some nonnegative 
random variable. Furthermore, suppose that the $X_j$ are identically distributed. Let us consider the characteristic function $\chi_{Z^{(k)}}$ of the random variable 
$Z^{(k)}:=\mu^{-1}\sum_{j=1}^{N_k}Z_j$. For $t\in\R$ we obtain

\begin{align}\label{ateq7}
\chi_{Z^{(k)}}(t)&=E[e^{itZ^{(k)}}]=E\bigl[E[e^{itZ^{(k)}}|N_k]\bigr]\notag\\
&=E\bigl[e^{-t^2s_{N_k}^2/2}\bigr]=E\bigl[e^{-\frac{t^2\sigma^2}{2}\frac{N_k}{\mu_k}}\bigr]\,.
\end{align}
  
Since the mapping $[0,\infty)\ni x\mapsto e^{-\frac{t^2\sigma^2x}{2}}\in\R$ is bounded and continuous for each $t\in\R$ we conclude that 

\begin{equation}\label{ateq8}
\chi_{Z^{(k)}}(t)\stackrel{k\to\infty}{\longrightarrow}E\bigl[e^{-\frac{t^2\sigma^2}{2}U}\bigr]=:\psi(t)\,.
\end{equation}

Note that $\psi(1)=1$ and that $\psi$ is continuous at $0$ by the dominated convergence theorem and hence, is a characteristic 
function of some random variable $Y$. It is easy to see by characteristic functions that $Y\stackrel{\D}{=}\sigma\sqrt{U}\zeta$, where $\zeta\sim N(0,1)$ is independent of $U$. Thus, we can conclude that 

\begin{equation}\label{ateq9}
Z^{(k)}\stackrel{\D}{\rightarrow}Y:=\sigma\sqrt{U}\zeta\text{  as }k\to\infty\,.
\end{equation}

If we also let $W^{(k)}:=\mu_k^{-1}\sum_{j=1}^{N_k} X_j$ and if we assume that the bounds from Corollary \ref{atcor1} (with $Z$ replaced by $Z^{(k)}$ and $W$ by $W^{(k)}$) tend to zero as $k$ goes to infinity, we can conclude that also 

\begin{equation}\label{ateq10}
W^{(k)}\stackrel{\D}{\rightarrow}Y=\sigma\sqrt{U}\zeta\text{  as }k\to\infty\,.
\end{equation}

The distribution $\calL(Y)$ is a scale mixture of the standard normal distribution and the result \eqref{ateq10} is in accordance with the theory from Chapter 3 of \cite{GneKo}, see for instance Theorem 3.3.2.

Now, the next task is to derive concrete error bounds on the distance\\ $d_{\calH}\bigl(\calL(Z),\calL(Y)\bigr)$ 
from bounds on the distance $d_{\calH}\bigl(\calL(\mu^{-1}N),\calL(U)\bigr)$. This appears to be possible, since in the case of identically distributed $X_j$ we have with $\zeta$ independent of $N$ that $Z\stackrel{\D}{=}\sigma\sqrt{\mu^{-1}N}\zeta$, as is easily checked. We will provide two propositions for the iid case, the first bounding the Kolmogorov distance of $\calL(Z)$ and $\calL(Y)$ in terms of $d_\K(\calL(\mu^{-1}N),\calL(U))$  and the second giving bounds 
on the Wasserstein distance of $\calL(Z)$ and $N(0,\sigma^2)$ in the case that $U=1$ is constant. Then we will turn to the case of not necessarily identically distributed 
summands $X_j$.

\begin{prop}\label{atprop1}
Let $\zeta\sim N(0,1)$ be independent of $N$ and of $U$. Then, for every $\sigma>0$
\[d_{\K}\bigl(\calL(\sigma\sqrt{\mu^{-1}N}\zeta),\calL(\sigma\sqrt{U}\zeta)\bigr)\leq d_\K(\calL(\mu^{-1}N),\calL(U))\,.\]
\end{prop}   

\begin{proof}
For $z\in\R$ we have 

\begin{align}\label{ateq11}
&\bigl|P(\sigma\sqrt{U}\zeta\leq z)-P(\sigma\sqrt{\mu^{-1}N}\zeta\leq z)\bigr|\notag\\
&=\Bigl|E\bigl[P(\sigma\sqrt{U}\zeta\leq z|\zeta)- P(\sigma\sqrt{\mu^{-1}N}\zeta\leq z|\zeta)\bigl]\Bigr|\notag\\
&\leq E\Bigl|P(\sigma\sqrt{U}\zeta\leq z|\zeta)- P(\sigma\sqrt{\mu^{-1}N}\zeta\leq z|\zeta)\Bigr|\,.
\end{align}

Now, for each $s\in\R$ by the independence of $\zeta$ and $U$ and of $\zeta$ and $N$

\begin{align}\label{ateq12}
&\Bigl|P(\sigma\sqrt{U}\zeta\leq z|\zeta=s)- P(\sigma\sqrt{\mu^{-1}N}\zeta\leq z|\zeta=s)\Bigr|\notag\\
&=\begin{cases}
\Bigl|P\bigl(\sqrt{U}\leq \frac{z}{\sigma s}\bigr)-P\bigl(\sqrt{\mu^{-1}N}\leq\frac{z}{\sigma s}\bigr)\Bigr|, &s>0\\
0, &s=0\\
\Bigl|P\bigl(\sqrt{U}\geq \frac{z}{\sigma s}\bigr)-P\bigl(\sqrt{\mu^{-1}N}\geq\frac{z}{\sigma s}\bigr)\Bigr|, &s<0
\end{cases}\notag\\
&\leq d_\K\bigl(\calL(\sqrt{U}),\calL(\sqrt{\mu^{-1}N})\bigr)\,,
\end{align}

since the Kolmogorov distance is also induced by the test functions $1_{[t,\infty)}$, $t\in\R$. Now, recall that the random variables $\mu^{-1}N$ and $U$ are nonnegative and, hence, for each $t\geq0$ we have 

\begin{align*}
\bigl|P(\sqrt{U}\leq t)-P(\sqrt{\mu^{-1}N}\leq t)\bigr|&=\bigl|P(U\leq t^2)-P(\mu^{-1}N\leq t^2)\bigr|\notag\\
&\leq d_\K\bigl(\calL(U),\calL(\mu^{-1}N)\bigr)\,,
\end{align*}

and taking the supremum over $t\geq0$ yields 

\begin{equation}\label{ateq13}
d_\K\bigl(\calL(\sqrt{U}),\calL(\sqrt{\mu^{-1}N})\bigr)\leq  d_\K\bigl(\calL(U),\calL(\mu^{-1}N)\bigr)\,.
\end{equation} 

Thus, from \eqref{ateq11}, \eqref{ateq12} and \eqref{ateq13} we obtain for every $z\in\R$ 

\begin{align*}
\bigl|P(\sigma\sqrt{U}\zeta\leq z)-P(\sigma\sqrt{\mu^{-1}N}\zeta\leq z)\bigr|&\leq E\bigl[d_\K\bigl(\calL(U),\calL(\mu^{-1}N)\bigr)\bigr]\notag\\
&=d_\K\bigl(\calL(U),\calL(\mu^{-1}N)\bigr)\,,
\end{align*}

proving the claim.\\
\end{proof}

\begin{prop}\label{atprop2}
If $\zeta\sim N(0,1)$ is independent of $N$, then for every $\sigma>0$ 

\[d_\W\bigl(\calL(\sigma\sqrt{\mu^{-1}N}\zeta), N(0,\sigma^2)\bigr)
\leq\sigma\sqrt{\frac{2}{\pi}}\frac{\sqrt{\Var(N)}}{\mu}\,.\]

\end{prop}
\begin{remark}\label{atrem2}
Of course, the conclusion of Proposition \ref{atprop2} can only yield useful bounds, if normal approximation of $\calL(W)$ makes sense. The upper bound 
shows that $\Var(N)$ should be of smaller order than $\mu^2=(E[N])^2$ in order for normal approximation to be plausible. 
\end{remark}

Proposition \ref{atprop2} immediately follows from the following two lemmas.

\begin{lemma}\label{atle3}
If $U\geq0$ and $\zeta\sim N(0,1)$ is independent of $U$ and of $N$, then

\[d_\W\bigl(\calL(\sigma\sqrt{U}\zeta),\calL(\sigma\sqrt{\mu^{-1}N}\zeta)\bigr)
\leq\sigma\sqrt{\frac{2}{\pi}}d_\W\bigl(\calL(\sqrt{U}),\calL(\sqrt{\mu^{-1}N})\bigr)\,.\]
\end{lemma}

\begin{proof}
Let $h:\R\rightarrow\R$ be $1$-Lipschitz. Then, 

\begin{align}\label{ateq14}
 &\bigl|E[h(\sigma\sqrt{U}\zeta)]-E[h(\sigma\sqrt{\mu^{-1}N}\zeta)]\bigr|\notag\\
&=\Bigl|E\bigl[E[h(\sigma\sqrt{U}\zeta)|\zeta]-E[h(\sigma\sqrt{\mu^{-1}N}\zeta)|\zeta]\bigr]\Bigr|\notag\\
&\leq E\Bigl|E[h(\sigma\sqrt{U}\zeta)|\zeta]-E[h(\sigma\sqrt{\mu^{-1}N}\zeta)|\zeta]\Bigr|\notag\\
&=E\Bigl|E\bigl[h(\sigma\sqrt{U}\zeta)-h(\sigma\sqrt{\mu^{-1}N}\zeta)\,\bigl|\,\zeta\bigr]\Bigr|
\end{align}

Now, for every $s\in\R$ by independence we have 

\begin{align}\label{ateq15}
& \Bigl|E\bigl[h(\sigma\sqrt{U}\zeta)-h(\sigma\sqrt{\mu^{-1}N}\zeta)\,\bigl|\,\zeta=s\bigr]\Bigr|\notag\\
&=\Bigl|E\bigl[h(\sigma s\sqrt{U})-h(\sigma s\sqrt{\mu^{-1}N})\bigr]\Bigr|\notag\\
&\leq \abs{s}\sigma d_\W\bigl(\calL(\sqrt{U}),\calL(\sqrt{\mu^{-1}N})\bigr)\,.
\end{align}
 
Hence, from \eqref{ateq14} and \eqref{ateq15} we obtain that 

\begin{align*}
 &\bigl|E[h(\sigma\sqrt{U}\zeta)]-E[h(\sigma\sqrt{\mu^{-1}N}\zeta)]\bigr|\\
&\leq \sigma E\bigl[\abs{\zeta}\bigr] d_\W\bigl(\calL(\sqrt{U}),\calL(\sqrt{\mu^{-1}N})\bigr) \\
&=\sigma\sqrt{\frac{2}{\pi}}d_\W\bigl(\calL(\sqrt{U}),\calL(\sqrt{\mu^{-1}N})\bigr)\,,
\end{align*}

as claimed.

\end{proof}

\begin{lemma}\label{atle4}
 Suppose that $U=1$ is constant. Then, 

\[d_\W\bigl(\calL(\sqrt{U}),\calL(\sqrt{\mu^{-1}N})\bigr)\leq\frac{\sqrt{\Var(N)}}{\mu}\,.\]

\end{lemma}

\begin{proof}
For a nonnegative real number $x$ we have the inequality

\begin{equation}\label{ateq16}
 \abs{1-x}=\frac{\abs{1-x^2}}{1+x}\leq\abs{1-x^2}\,.
\end{equation}
 
From \eqref{ateq16} with $x=\sqrt{\mu^{-1}N}$ and for each $1$-Lipschitz function $h$ we obtain

\begin{align}\label{ateq17}
 \Bigl|E\bigl[h(\sqrt{\mu^{-1}N})-h(1)\bigr]\Bigr|&\leq E\bigl|1-\sqrt{\mu^{-1}N}\bigr|\notag\\
&\leq E\bigl|1-\mu^{-1}N\bigr|\leq\sqrt{\Var(\mu^{-1}N)}\notag\\
&= \frac{\sqrt{\Var(N)}}{\mu}\,.
\end{align}

\end{proof}

Now, we turn to the general case that the $X_j$ are no more supposed to be identically distributed. Here, we will assume that 

\begin{equation}\label{ateq18}
 0<\hat{\sigma}^2:=\lim_{n\to\infty}\frac{1}{n}\sum_{j=1}^n \sigma_j^2<\infty
\end{equation}

exists. Note that in the iid case $\hat{\sigma}^2=\sigma^2=\frac{1}{n}\sum_{j=1}^n \sigma_j^2$ for each $n\geq1$. We further define 

\begin{equation}\label{sigmanhat}
 \hat{\sigma}_n^2:=\sum_{j=1}^n \sigma_j^2\,.
\end{equation}

Note that we have the general representation

\begin{equation}\label{repzgen}
Z\stackrel{\D}{=}\sqrt{\mu^{-1} N}\hat{\sigma}_n \zeta\,.
\end{equation}

The following lemma will be useful.

\begin{lemma}\label{atle5}
Let $\sigma,\tau\in(0,\infty)$. Then, 
\begin{enumerate}[{\normalfont(a)}]
 \item $d_\K\bigl(N(0,\sigma^2),N(0,\tau^2)\bigr)\leq\min\bigl(\abs{1-\frac{\tau^2}{\sigma^2}},\abs{1-\frac{\sigma^2}{\tau^2}}\bigr)$ and 
 \item $d_\W\bigl(N(0,\sigma^2),N(0,\tau^2)\bigr)\leq 2\min\bigl(\sigma\abs{1-\frac{\tau^2}{\sigma^2}},\tau\abs{1-\frac{\sigma^2}{\tau^2}}\bigr)$.
\end{enumerate}
\end{lemma}

\begin{proof}
 The proof uses Stein's method for univariate normal approximation. Let $\zeta\sim N(0,1)$. We first prove (a). Let $t\in\R$ be arbitrary. Then, 

\begin{align}\label{ateq19}
\bigl|P(\sigma\zeta\leq t)-P(\tau\zeta\leq t)\bigr|&=\bigl|P(\zeta\leq\frac{t}{\sigma})-P(\frac{\tau}{\sigma}\zeta\leq\frac{t}{\sigma})\bigr|\notag\\
&\leq d_\K\bigl(N(0,1),N(0,\frac{\tau^2}{\sigma^2})\bigr)\,. 
\end{align}
 
Hence, by symmetry,  

\begin{equation}\label{ateq20}
 d_K\bigl(N(0,\sigma^2),N(0,\tau^2)\bigr)\leq\max\Bigl(d_\K\bigl(N(0,1),N(0,\frac{\tau^2}{\sigma^2})\bigr),d_\K\bigl(N(0,1),N(0,\frac{\sigma^2}{\tau^2})\bigr)\Bigr)\,.
\end{equation}

Now, for $z\in\R$ let $f_z$ be the standard solution to \textit{Stein's equation} 

\begin{equation}\label{steineqz}
f'(x)-xf(x)=1_{(-\infty,z]}-P(\zeta\leq z)\,,
\end{equation}

which is given by $f_z(x)=\frac{\Phi(z\wedge x)-\Phi(x)\Phi(z)}{\phi(x)}$,where $\phi$ and $\Phi$ denote the density function and the cumulative distribution function of $N(0,1)$, respectively.
It is known that $f_z$ is Lipschitz with $\fnorm{f_z'}\leq1$ for each $z\in\R$ and that for each $c>0$ and every Lipschitz function $f$ it holds that 
$E[c\zeta f(c\zeta)]=c^2 E[f'(c\zeta)]$. Hence, by inserting 
$c\zeta$ into \eqref{steineqz} and taking expectations

\begin{align*}
\bigl|P(c\zeta\leq z)-P(\zeta\leq z)\bigr|&=\bigl|E[f_z'(c\zeta)-c\zeta f_z(c\zeta)]\notag\\
&=\abs{1-c^2} \abs{E[f_z'(c\zeta)]}\leq \abs{1-c^2}\,,
\end{align*}

yielding 

\[d_\K\bigl(N(0,1),N(0,c^2)\bigr)\Bigr)\leq\abs{1-c^2}\,.\]

By \eqref{ateq20} this implies (a).\\

The proof of (b) is similar by using the solution to Stein's equation for a $1$-Lipschitz test function $h$. One merely has to take care of the scaling properties 
of Wasserstein distance in the equality corresponding to \eqref{ateq19} and note that in this case $\fnorm{f_h'}\leq 2$ for the standard solution $f_h$ of Stein's 
equation (see, e.g. \cite{CGS}).   

\end{proof}

The strategy to keep track of the non-identically distributed case is to use the inequality 

\begin{equation}\label{ateq21}
d_{\calH}\bigl(\calL(Z),\calL(\hat{Y})\bigr)\leq d_{\calH}\bigl(\calL(Z),\calL(\hat{Z})\bigr)+d_{\calH}\bigl(\calL(\hat{Z}),\calL(\hat{Y})\bigr)\,,
\end{equation}

where $\hat{Z}:=\hat{\sigma}\sqrt{\mu^{-1}N}\zeta$ and $\hat{Y}:=\hat{\sigma}\sqrt{U}\zeta$. 
The first term on the right hand side of \eqref{ateq21} may be bounded with the help of Lemma \ref{atle5} (see Lemma \ref{atle6}) and the second term may be handled by 
Proposition \ref{atprop1} in the case of non-constant $U$ and by Proposition \ref{atprop2}, if $U=1$ is constant and normal approximation is plausible. 
We will need one more lemma, comparing $\calL(Z)$ and $\calL(\hat{Z})$.

\begin{lemma}\label{atle6}
With the above definitions we have:
\begin{enumerate}[{\normalfont (a)}]
 \item $d_\K\bigl(\calL(Z),\calL(\hat{Z})\bigr)\leq \sum_{n=1}^\infty P(N=n)
\min\Bigl(\bigl|1-\frac{\hat{\sigma}^2}{\hat{\sigma}_n^2}\bigr|,\bigl|1-\frac{\hat{\sigma}_n^2}{\hat{\sigma}^2}\bigr|\Bigr)$.
\item $d_\W\bigl(\calL(Z),\calL(\hat{Z})\bigr)\leq 2\mu^{-1/2}\sum_{n=1}^\infty \sqrt{n}P(N=n)
\min\Bigl(\hat{\sigma}_n\bigl|1-\frac{\hat{\sigma}^2}{\hat{\sigma}_n^2}\bigr|,\hat{\sigma}\bigl|1-\frac{\hat{\sigma}_n^2}{\hat{\sigma}^2}\bigr|\Bigr)$.
\end{enumerate}
\end{lemma}

\begin{proof}
Let $\calH$ be either equal to $\K$ or to $\W$. Then, recalling representation \eqref{repzgen}, for $h\in\calH$

\begin{align}\label{ateq22}
 &\Bigl|E\bigl[h(Z)\bigr]-E\bigl[h(\hat{Z})\bigr]\Bigr|\notag\\
&\leq\sum_{n=1}^\infty P(N=n)\Bigl|E\bigl[h(\mu^{-1/2}\sqrt{n}\hat{\sigma}_n\zeta)\bigr]-E\bigl[h(\mu^{-1/2}\sqrt{n}\hat{\sigma}\zeta)\bigr]\Bigr|\notag\\
&\leq \sum_{n=1}^\infty P(N=n) d_{\calH}\Bigl(N(0,\mu^{-1}n\hat{\sigma}_n^2), N(0,\mu^{-1}n\hat{\sigma}^2)\Bigr)\,.
\end{align}
 
If $\calH=\K$, then the bound in \eqref{ateq22} equals 

\begin{equation*}
 \sum_{n=1}^\infty P(N=n) d_{\K}\Bigl(N(0,\hat{\sigma}_n^2), N(0,\hat{\sigma}^2)\Bigr)
\leq\sum_{n=1}^\infty P(N=n)
\min\Bigl(\bigl|1-\frac{\hat{\sigma}^2}{\hat{\sigma}_n^2}\bigr|,\bigl|1-\frac{\hat{\sigma}_n^2}{\hat{\sigma}^2}\bigr|\Bigr)
\end{equation*}

by Lemma \ref{atle5} (a). If $\calH=\W$, then the bound in \eqref{ateq22} equals 

\begin{align*}
&\mu^{-1/2}\sum_{n=1}^\infty \sqrt{n}P(N=n) d_{\W}\Bigl(N(0,\hat{\sigma}_n^2), N(0,\hat{\sigma}^2)\Bigr)\\
&\leq2\mu^{-1/2}\sum_{n=1}^\infty \sqrt{n}P(N=n)
\min\Bigl(\hat{\sigma}_n\bigl|1-\frac{\hat{\sigma}^2}{\hat{\sigma}_n^2}\bigr|,\hat{\sigma}\bigl|1-\frac{\hat{\sigma}_n^2}{\hat{\sigma}^2}\bigr|\Bigr)
\end{align*}

by Lemma \ref{atle5} (b). This completes the proof.

\end{proof}

Putting the pieces together, we have the following proposition.

\begin{prop}\label{atprop3}
Under the general assumptions and if \eqref{ateq18} is satisfied, we have
\begin{enumerate}[{\normalfont (a)}]
 \item $d_{\K}\bigl(\calL(Z),\calL(\hat{Y})\bigr)\leq \sum_{n=1}^\infty P(N=n)
\min\Bigl(\bigl|1-\frac{\hat{\sigma}^2}{\hat{\sigma_n}^2}\bigr|,\bigl|1-\frac{\hat{\sigma_n}^2}{\hat{\sigma}^2}\bigr|\Bigr)\\ + d_\K(\calL(\mu^{-1}N),\calL(U))$ and
 \item $d_{\W}\bigl(\calL(Z),N(0,\hat{\sigma}^2)\bigr)\leq 2\mu^{-1/2}\sum_{n=1}^\infty \sqrt{n}P(N=n)
\min\Bigl(\hat{\sigma_n}\bigl|1-\frac{\hat{\sigma}^2}{\hat{\sigma_n}^2}\bigr|,\hat{\sigma}\bigl|1-\frac{\hat{\sigma_n}^2}{\hat{\sigma}^2}\bigr|\Bigr)
+\hat{\sigma}\sqrt{\frac{2}{\pi}}\frac{\sqrt{\Var(N)}}{\mu}$,\\
where $\hat{\sigma}_n^2$ was defined in \eqref{sigmanhat}. 
\end{enumerate}
\end{prop}

\begin{proof}
The claims follow from Propositions \ref{atprop1} and \ref{atprop2} with $\sigma$ replaced by $\hat{\sigma}$ to deal with the respective second term on the right hand side 
of \eqref{ateq21} and from Lemma \ref{atle6} to handle the first term.
 
\end{proof}

\section{Applications}\label{apps}

\subsection{Geometric Sums}\label{geo}
In this subsection, $N=N_p$ will always have the geometric distribution on $\{1,2,\dotsc\}$ with parameter $p\in(0,1)$, i.e. $P(N=n)=p(1-p)^{n-1}$ 
for each integer $n\geq1$. Note that this implies $\mu:=\mu_p:=E[N_p]=p^{-1}$. In order to derive moore concrete rates of convergence for geometric sums from Corollary \ref{atcor1}, 
we will first prove the following lemma.

\begin{lemma}\label{geole1}
 Let $N_p$ have the geometric distribution with parameter $p\in(0,1)$. Then, 

\[\sqrt{p}\leq E\Bigl[\frac{1}{\sqrt{N_p}}\Bigr]\leq\frac{2\sqrt{p}}{1+\sqrt{p}}\leq2\sqrt{p}\,.\]

\end{lemma}

\begin{proof}
We will first prove the left-most inequality. Let $\phi:(0,\infty)\rightarrow(0,\infty)$ be defined by $\phi(x):=\frac{1}{\sqrt{x}}$. Then, $\phi$ is convex 
and hence, by Jensen's inequality, we have 

\[\sqrt{p}=\phi\bigl(E[N_p]\bigr)\leq E\bigl[\phi(N_p)\bigr]=E\Bigl[\frac{1}{\sqrt{N_p}}\Bigr]\,.\]

Now, note that the right-most inequality is trivial. For the remaining inequality we write 

\begin{align}\label{geoeq1}
 E\Bigl[\frac{1}{\sqrt{N_p}}\Bigr]&=\sum_{k=1}^\infty\frac{1}{\sqrt{k}}p(1-p)^{k-1}=\frac{p}{1-p}f(p)\,,
\end{align}
 
where $f(p):=\sum_{k=1}^\infty\frac{1}{\sqrt{k}}(1-p)^{k}$ for $0<p\leq1$. Note that $f$ is continuous and has an analytic extension to $(0,2)$. We have 

\[f'(p)=-\sum_{k=1}^\infty\sqrt{k}(1-p)^{k-1}=-\frac{1}{p}E\bigl[\sqrt{N_p}\bigr]\,.\] 

Thus, by the fundamental theorem of calculus, we have

\begin{align}\label{geoeq2}
f(p)&=f(1)+\int_1^p f'(s)ds =-\int_1^p  \frac{1}{s}E\bigl[\sqrt{N_s}\bigr]ds\notag\\
&=\int_p^1\frac{E\bigl[\sqrt{N_s}\bigr]}{s}ds\,.
\end{align}

Since $\sqrt{x}$ is concave on $(0,\infty)$, it follows again by Jensen's inequality  that 

\[E\bigl[\sqrt{N_s}\bigr]\leq\sqrt{E[N_s]}=\frac{1}{\sqrt{s}}\,.\]

Thus, form \eqref{geoeq2} we conclude that 

\begin{equation}\label{geoeq3}
f(p)\leq \int_p^1 s^{-3/2}ds=-2s^{-1/2}\Bigl|_p^1=2\Bigl(\frac{1}{\sqrt{p}}-1\Bigr) \,.
\end{equation}

Hence, from \eqref{geoeq1} and \eqref{geoeq3} we obtain 

\begin{equation*}
 E\Bigl[\frac{1}{\sqrt{N_p}}\Bigr]\leq\frac{p}{1-p}2\Bigl(\frac{1}{\sqrt{p}}-1\Bigr) =2\sqrt{p}\frac{1-\sqrt{p}}{1-p}=\frac{2\sqrt{p}}{1+\sqrt{p}}\,,
\end{equation*}

proving the inequality.

\end{proof}

\begin{remark}\label{georem1}
 Since the function $\sqrt{x}$ is concave, on could try to make use of the inequality 

\begin{equation*}
 E\Bigl[\frac{1}{\sqrt{N_p}}\Bigr]=E\Bigl[\sqrt{\frac{1}{N_p}}\Bigr]\leq\sqrt{E\bigl[\frac{1}{N_p}\bigr]}
\end{equation*}

for the upper bound. But as one can show that $E[N_p^{-1}]=\dfrac{-p\log p}{1-p}$, this direct application of Jensen's inequality would only yield the worse bound 

\[E\Bigl[\frac{1}{\sqrt{N_p}}\Bigr]\leq\sqrt{p}\Bigl(\frac{-\log p}{1-p}\Bigr)^{1/2}\,.\]

\end{remark}

From Corollary \ref{atcor1} and Lemma \ref{geole1} we obtain the following result.

\begin{prop}\label{geoprop1}
Let, additionally to the general assumptions, the random variables $X_j$ be identically distributed with $\sigma^2:=E[X_1]$ and $\xi:=E\abs{X_1}^3$, then 

\[d_\K\bigl(\calL(W),\calL(Z)\bigr)\leq\frac{2C_\K\xi}{\sigma^3}\sqrt{p}\,.\]

\end{prop}

It can easily be checked by means of characteristic functions that $pN_p\stackrel{\D}{\rightarrow}U\sim\Exp(1)$ as $p\downarrow 0$. 
From Example 3.1 of \cite{PeRoe} we even have the following quantitative version of this result: 

\begin{equation}\label{geoeq4}
d_\K\bigl(\calL(pN_p), \Exp(1)\bigr)\leq12 p
\end{equation}

From \eqref{ateq9} we see that we should approximate $\calL(W)$ by the distribution of the random variable $Y=\sigma\sqrt{U}\zeta$, where $\zeta\sim N(0,1)$ 
is independent of $U$. With the help of chatacteristic functions one can easily show that $Y$ has the Laplace distribution $\Laplace(0,\sigma/\sqrt{2})$, where for 
$a\in\R$ and $b>0$ the distribution $\Laplace(a,b)$ is defined by the density $f_{a,b}:\R\rightarrow\R$ with 

\begin{equation*}
f_{a,b}(x):=\frac{1}{2b}e^{-\frac{\abs{x-a}}{b}}\,.
\end{equation*}
 
Our observations now lead at once to the following theorem.

\begin{theorem}[Berry-Esseen bound for geometric sums of iid centered random variables]\label{geotheo1} 
Let $X_1,X_2,\dotsc$ be independent and identically distributed random variables with $E[X_1]=0$, $0<\sigma^2:=E[X_1^2]$ and $\xi:=E\abs{X_1}^3<\infty$ and 
let $N=N_p$ be independent of the $X_j$ and have the geometric distribution on $\{1,2,\dotsc\}$ with parameter $p\in(0,1)$. Then,

\[d_\K\Bigl(\calL(W),\Laplace\bigl(0,\frac{\sigma}{\sqrt{2}}\bigr)\Bigr)\leq\frac{2C_\K\xi}{\sigma^3}\sqrt{p} +12p\,.\]

\end{theorem}

\begin{proof}
This follows from Proposition \ref{atprop1}, Proposition \ref{geoprop1} and \eqref{geoeq4} by the triangle inequality for $d_\K$. 

\end{proof}

Now, we turn to not necessarily identically distributed summands $X_j$. Let in accordance with $\hat{\sigma}_n$ the quantities $\hat{\xi}_n$ be defined by 

\begin{equation}\label{geoeq5}
 \hat{\xi}_n:=\frac{1}{n}\sum_{j=1}^n \xi_j\,.
\end{equation}

Then, we can prove the following theorem.

\begin{theorem}[Berry-Esseen bound for geometric sums of independent, centered random variables]\label{geotheo2}  
Let $X_1,X_2,\dotsc$ be independent random variables with $E[X_j]=0$, $0<\sigma_j^2:=E[X_j^2]$ and $\xi_j:=E\abs{X_j}^3<\infty$ for $j\geq1$ and 
let $N=N_p$ be independent of the $X_j$ and have the geometric distribution on $\{1,2,\dotsc\}$ with parameter $p\in(0,1)$. Assume that \eqref{ateq18} is satisfied. 
Then, with $\hat{\sigma}_n^2$ defined by \eqref{sigmanhat} and with $\hat{\xi}_n$ defined by \eqref{geoeq5} it holds that 

\begin{align*}
 d_\K\Bigl(\calL(W),\Laplace\bigl(0,\frac{\hat{\sigma}}{\sqrt{2}}\bigr)\Bigr)&\leq 12p + p\sum_{n=1}^\infty (1-p)^{n-1}
\min\Bigl(\bigl|1-\frac{\hat{\sigma}^2}{\hat{\sigma}_n^2}\bigr|,\bigl|1-\frac{\hat{\sigma}_n^2}{\hat{\sigma}^2}\bigr|\Bigr)\\ 
&\;+C_\K E\Bigl[\frac{1}{\sqrt{N}}\frac{\hat{\xi}_N}{\hat{\sigma}_N^3}\Bigr]\,.
\end{align*}
\end{theorem}

\begin{proof}
Note that since $s_N^2=\mu^{-1}N\hat{\sigma}_N^2$, by Proposition \ref{attheo1} (b)

\[ d_\K\bigl(\calL(W),\calL(Z)\bigr)\leq\frac{C_\K}{\mu^{3/2}}E\Bigl[\frac{1}{s_N^3}\sum_{j=1}^N \xi_j\Bigr]=C_\K E\Bigl[\frac{1}{\sqrt{N}}\frac{\hat{\xi}_N}{\hat{\sigma}_N^3}\Bigr]\,.\]

Furthermore, by Proposition \ref{atprop3} (a) and \eqref{geoeq4}

\begin{align*}
d_{\K}\bigl(\calL(Z),\Laplace\bigl(0,\frac{\hat{\sigma}}{\sqrt{2}}\bigr)\leq \sum_{n=1}^\infty p(1-p)^{n-1}
\min\Bigl(\bigl|1-\frac{\hat{\sigma}^2}{\hat{\sigma_n}^2}\bigr|,\bigl|1-\frac{\hat{\sigma_n}^2}{\hat{\sigma}^2}\bigr|\Bigr) + 12p\,.
\end{align*} 

Hence, the claim follows from the triangle inequality.

\end{proof}

\begin{remark}\label{georem2}
A different way to prove Berry-Esseen bounds in the non-identically distributed setting is to make use of the bound 

\begin{equation}\label{geoeq6}
 d_{\K}\Bigl(\calL(s_N^2),\Exp\bigl(\frac{1}{E[\sigma_N^2]}\bigr)\Bigr)\leq \frac{12p}{E[\sigma_N^2]}\sup_{j\geq1}\sigma_j^2\,, 
\end{equation}

which follows from Theorem 3.1 in \cite{PeRoe}, noting that $E[s_N^2]=E[\sigma_N^2]$ in case of the geometric distribution. This, in turn, follows from \eqref{ateq2} 
because $P(N\geq j)=(1-p)^{j-1}$ and $\mu^{-1}=p$. Then, by a suitable Proposition, which is completely analogous to Proposition \ref{atprop1}, one can show that since $Z\stackrel{\D}{=}s_N\zeta$

\begin{equation}\label{geoeq7}
d_\K\Bigl(\calL(Z),\Laplace\bigl(0,(E[\sigma_N^2])^{1/2}/\sqrt{2}\bigr)\Bigr)\leq \frac{12p}{E[\sigma_N^2]}\sup_{j\geq1}\sigma_j^2\,.
\end{equation}

Thus, we have the following alternative bound in the situation of Theorem \ref{geotheo2} (without assuming the existence of $\hat{\sigma}^2$):

\begin{equation}\label{geoeq8}
 d_\K\Bigl(\calL(W),\Laplace\bigl(0,(E[\sigma_N^2])^{1/2}/\sqrt{2}\bigr)\Bigr)\leq C_\K E\Bigl[\frac{1}{\sqrt{N}}\frac{\hat{\xi}_N}{\hat{\sigma}_N^3}\Bigr] 
+ \frac{12p}{E[\sigma_N^2]}\sup_{j\geq1}\sigma_j^2
\end{equation}

Note that both, the bound from Theorem \ref{geotheo2} and \eqref{geoeq8} reduce to the bound from Theorem \ref{geotheo1} in the iid case.
\end{remark}

\subsection{The Poisson case}\label{poisson}
In this subsection, we discuss the case that $N=N_\la\sim\Poi(\la)$ with $\la>0$. Note that this implies $E[N]=\Var(N)=\la$ and, hence, by Remark \ref{atrem2} 
normal approximation of the random sum indexed by $N$ is appropriate, as $\la\to\infty$. Again, we begin with the identically distributed case.

\begin{theorem}\label{potheo1}
Let $X_1,X_2,\dotsc$ be independent and identically distributed random variables with $E[X_1]=0$, $0<\sigma^2:=E[X_1^2]$ and $\xi:=E\abs{X_1}^3<\infty$ and let 
$N\sim\Poi(\la)$, $\la>0$, be independent of the $X_j$. Then, 

\[d_\W\bigl(\calL(W),N(0,\sigma^2)\bigr)\leq\Bigl(\frac{C_\W\xi}{\sigma^2}+\sigma\sqrt{\frac{2}{\pi}}\Bigr)\frac{1}{\sqrt{\la}}\,.\]
\end{theorem}

\begin{proof}
From Corollary \ref{atcor1} (a) we have 

\[d_\W\bigl(\calL(W),\calL(Z)\bigr)\leq \frac{C_\W\xi}{\sqrt{\la}\sigma^2}\]

and from Proposition \ref{atprop2} it follows that 

\[d_\W\bigl(\calL(Z),N(0,\sigma^2)\bigr)\leq \sigma\sqrt{\frac{2}{\pi}}\frac{1}{\sqrt{\la}}\,.\]

Hence, the assertion follows from the triangle inequality for $d_\W$.

\end{proof}

The next result generalizes Theorem \ref{potheo1} to non-identically distributed summands.

\begin{theorem}\label{potheo2}
Let $X_1,X_2,\dotsc$ be independent random variables with $E[X_j]=0$, $0<\sigma_j^2:=E[X_j^2]$ and $\xi_j:=E\abs{X_j}^3<\infty$, $j\geq1$, and let 
$N\sim\Poi(\la)$, $\la>0$, be independent of the $X_j$. Assume that \eqref{ateq18} is satisfied. Then, with $\hat{\sigma}_n^2$ defined by \eqref{sigmanhat} and with $\hat{\xi}_n$ defined by \eqref{geoeq5} it holds that  

\begin{align*}
d_\W\bigl(\calL(W),N(0,\hat{\sigma}^2)\bigr)&\leq\hat{\sigma}\sqrt{\frac{2}{\pi}}\frac{1}{\sqrt{\la}}
+\frac{C_\W}{\sqrt{\la}} E\Bigl[\frac{\hat{\xi}_N}{\hat{\sigma}_N^2}\Bigr]\\
&\;+\frac{2e^{-\la}}{\sqrt{\la}}\sum_{n=1}^\infty \frac{\la^{n}}{\sqrt{n}(n-1)!}
\min\Bigl(\hat{\sigma}_n\bigl|1-\frac{\hat{\sigma}^2}{\hat{\sigma}_n^2}\bigr|,\hat{\sigma}\bigl|1-\frac{\hat{\sigma}_n^2}{\hat{\sigma}^2}\bigr|\Bigr)\,.
\end{align*}
\end{theorem}

\begin{proof}
From Proposition \ref{attheo1} (a) it follows that 

\[ d_\W\bigl(\calL(W),\calL(Z))\leq \frac{C_\W}{\mu^{3/2}}E\Bigl[\frac{1}{s_N^2}\sum_{j=1}^N \xi_j\Bigr]=\frac{C_\W}{\sqrt{\la}} E\Bigl[\frac{\hat{\xi}_N}{\hat{\sigma}_N^2}\Bigr]\,,\]

since $s_N^2=\la^{-1}\hat{\sigma}_N^2$. From Proposition \ref{atprop3} (b) it follows by some minor calculations that 

\[d_{\W}\bigl(\calL(Z),N(0,\hat{\sigma}^2)\bigr)\leq\frac{2e^{-\la}}{\sqrt{\la}}\sum_{n=1}^\infty \frac{\la^{n}}{\sqrt{n}(n-1)!}
\min\Bigl(\hat{\sigma}_n\bigl|1-\frac{\hat{\sigma}^2}{\hat{\sigma}_n^2}\bigr|,\hat{\sigma}\bigl|1-\frac{\hat{\sigma}_n^2}{\hat{\sigma}^2}\bigr|\Bigr)
+\hat{\sigma}\sqrt{\frac{2}{\pi}}\frac{1}{\sqrt{\la}}\,. \]

Again, the claim follows from the triangle inequality for $d_\W$.

\end{proof}

\subsection{The binomial case}\label{bin}

In this subsection we let $N=N_m\sim\Bin(m,p)$ where $m\geq1$ is an integer, which is supposed to be large, and $p\in(0,1)$ is considered fixed. Note that this implies 
$E[N]=mp$ and $\Var(N)=mp(1-p)$. Thus, by Remark \ref{atrem2} normal approximation of the random sum indexed by $N$ is plausible, as $m\to\infty$. 
As usual, we first state our theorem for identically distributed summands.

\begin{theorem}\label{bintheo1}
 Let $X_1,X_2,\dotsc$ be independent and identically distributed random variables with $E[X_1]=0$, $0<\sigma^2:=E[X_1^2]$ and $\xi:=E\abs{X_1}^3<\infty$ and let 
$N\sim\Bin(m,p)$ be independent of the $X_j$. Then, 

\[d_\W\bigl(\calL(W),N(0,\sigma^2)\bigr)\leq\Bigl(\frac{C_\W\xi}{\sigma^2}+\sigma\sqrt{\frac{2(1-p)}{\pi}}\Bigr)\frac{1}{\sqrt{mp}}\,.\]

\end{theorem}

\begin{proof}
The proof, being completely analogous to that of Theorem \ref{potheo1}, is omitted.

\end{proof}

\begin{remark}\label{binrem1}
\begin{enumerate}[(i)]
\item The bound in Theorem \ref{bintheo1} shows that normal approximation is appropriate even if $p$ goes to zero as long as $mp$ goes to infinity. This includes for instance 
the cases $p=p_m=cm^{-\alpha}$ for some constant $c$ and $0<\alpha<1$.
\item Note that for $p=1$ the bound from Theorem \ref{bintheo1} equals the Berry-Esseen bound in the Wasserstein distance for 
a sum of $m$ random variables. This makes sense since a random variable $N\sim\Bin(m,1)$ equals $m$ almost surely. 
\end{enumerate}
\end{remark}

The next theorem generalizes the last result to the case of not necessarily identically distributed summands.

\begin{theorem}\label{bintheo2}
 Let $X_1,X_2,\dotsc$ be independent random variables with $E[X_j]=0$, $0<\sigma_j^2:=E[X_j^2]$ and $\xi_j:=E\abs{X_j}^3<\infty$, $j\geq1$, and let 
$N\sim\Bin(m,p)$ be independent of the $X_j$. Assume that \eqref{ateq18} is satisfied. Then, with $\hat{\sigma}_n^2$ defined by \eqref{sigmanhat} and with $\hat{\xi}_n$ defined by \eqref{geoeq5} it holds that  

\begin{align*}
&d_\W\bigl(\calL(W),N(0,\hat{\sigma}^2)\bigr)\leq\Bigl(\hat{\sigma}\sqrt{\frac{2(1-p)}{\pi}}
+C_\W E\Bigl[\frac{\hat{\xi}_N}{\hat{\sigma}_N^2}\Bigr]\Bigr)\frac{1}{\sqrt{mp}}\\
&+\;\frac{2}{\sqrt{mp}}\sum_{n=1}^m\sqrt{n}\binom{m}{n}p^{n}(1-p)^{m-n} 
\min\Bigl(\hat{\sigma}_n\bigl|1-\frac{\hat{\sigma}^2}{\hat{\sigma}_n^2}\bigr|,\hat{\sigma}\bigl|1-\frac{\hat{\sigma}_n^2}{\hat{\sigma}^2}\bigr|\Bigr)\,.
\end{align*}
\end{theorem}

\begin{proof}
 The proof, being completely analogous to that of Theorem \ref{potheo2}, is omitted.

\end{proof}

\section*{Acknowledgements}
I am grateful to Peter Eichelsbacher for pointing me to the topic of random sums and I thank Mathias Rafler for a lot of fruitful conversation.
\normalem
\bibliography{randomsums}{}
\bibliographystyle{alpha}

\end{document}